\documentclass[12pt]{article}
\usepackage{amsmath,amssymb}
\usepackage{fancyhdr}  
\usepackage[a4paper, margin=2.6cm]{geometry}

\usepackage[colorlinks=true,bookmarks=false,linkcolor=blue,citecolor=blue,urlcolor=blue]{hyperref}
\hypersetup{
pdftitle={Tightness of Chekanov's bound on displacement for some Lagrangian knots}
}
\usepackage{bookmark}
\bookmarksetup{
  numbered
}
\usepackage{xspace}
\usepackage{blindtext}
\usepackage{dsfont}
\usepackage{mathtools}
\usepackage{float}
\usepackage[dvipsnames]{xcolor}
\usepackage{mathabx}
\usepackage{tikz-cd}
\usepackage{tikz}
\usetikzlibrary{decorations.pathreplacing}
\usepackage{tikz-3dplot}
\usetikzlibrary{math}
\usetikzlibrary{patterns}
\usepackage{mathrsfs}
\usetikzlibrary{calc,decorations.markings}
\usepackage{caption}
\usepackage{subcaption}
\usepackage[
minalphanames=10,
maxalphanames=10,
backend=biber,
style=alphabetic
]{biblatex}

\definecolor{darkgreen}{rgb}{0,0.7,0}

\usetikzlibrary{decorations.markings}
\tikzset{double line with arrow/.style args={#1,#2}{decorate,decoration={markings,%
mark=at position 0 with {\coordinate (ta-base-1) at (0,1pt);
\coordinate (ta-base-2) at (0,-1pt);},
mark=at position 1 with {\draw[#1] (ta-base-1) -- (0,1pt);
\draw[#2] (ta-base-2) -- (0,-1pt);
}}}}

\usepackage{mathabx,epsfig}

\definecolor{electricviolet}{rgb}{0.56, 0.0, 1.0}

\usepackage{etoolbox}
\usepackage{amsthm}

\newtheorem{introthm}{Theorem}

\newtheorem{introquestion}{Question}
\newtheorem{introprop}{Proposition}

\newtheorem{theorem}{Theorem}[section]
\newtheorem{question}[theorem]{Question}

\newtheorem*{theorem*}{Theorem}
\newtheorem*{lemma*}{Lemma}
\newtheorem{lemma}{Lemma}[section]
\newtheorem*{definition*}{Definition}
\newtheorem{definition}{Definition}
\newtheorem{corollary}[theorem]{Corollary}

\newtheorem*{fact*}{Fact}
\newtheorem{remark}[theorem]{Remark}
\newtheorem*{remark*}{Remark}
\newtheorem{prop}[theorem]{Proposition}
\newtheorem*{proposition*}{Proposition}
\newtheorem{example}[theorem]{Example}

\newcommand{\C}{\mathbb{C}}

\newcommand{\scomp}[2]{\left\{{#1}\,\middle|\,{#2}\right\}}
\newcommand{\overbar}[1]{\mkern 1.5mu\overline{\mkern-1.5mu#1\mkern-1.5mu}\mkern 1.5mu}
\newcommand{\conj}[1]{\overbar{#1}}

\DeclareMathOperator{\IdMap}{id}

\DeclareMathOperator{\vol}{vol}

\newcommand{\interior}[1]{\ring{#1}}
\newcommand{\invH}{{\hbar_D}}
\newcommand{\invS}{\hbar_S}

\usepackage{datetime}
\newdateformat{monthyeardate}{\monthname[\THEMONTH] \THEYEAR}
\newdateformat{monthdayyeardate}{\monthname[\THEMONTH] \THEDAY, \THEYEAR}

\begin{document}

\title{Tightness of Chekanov's bound on displacement energy for some Lagrangian knots\thanks{To appear in the \emph{Journal of Symplectic Geometry}.}}

\author{
   David Keren Yaar
}

\date{}

\maketitle

\renewcommand{\thefootnote}{}


\renewcommand{\thefootnote}{\arabic{footnote}}

\begin{abstract}
By a classical theorem of Chekanov, the displacement energy, \texorpdfstring{$e$}{e}, of a Lagrangian submanifold is bounded from below by the minimal area, \texorpdfstring{$\hbar$}{ℏ}, of pseudo-holomorphic disks with boundary on the Lagrangian. We compute \texorpdfstring{$e$}{e} and \texorpdfstring{$\hbar$}{ℏ} for displaceable Chekanov tori in \texorpdfstring{$\mathbb{C}P^n$}{ℂPⁿ}, and for an infinite family of exotic tori in \texorpdfstring{$\mathbb{C}^3$}{ℂ³} constructed by Brendel. In these families, \texorpdfstring{$e=\hbar$}{e=ℏ}.

We compare continuity properties of \texorpdfstring{$e$}{e} and \texorpdfstring{$\hbar$}{ℏ} on the space of Lagrangians. This provides an example (suggested by Fukaya, Oh, Ohta, and Ono) where \texorpdfstring{$e>\hbar$}{e>ℏ}.

Our calculations have further applications such as a new proof, inspired by work of Auroux, that Brendel's family of exotic tori consists of infinitely many distinct Lagrangians.
\end{abstract}

\tableofcontents

\newpage

\section{Introduction}

Let $(M,\omega)$ be a symplectic manifold and $L\subset M$ be a closed embedded Lagrangian submanifold. In the present note, we compare the following two fundamental notions in symplectic geometry, which quantify the size of $L$, through concrete examples.

The first notion is Hofer's displacement energy. Denote the \emph{Hofer norm} by \[\lVert H\rVert\coloneqq\int_0^1\max_MH_t-\min_MH_tdt\] where $H\colon[0,1]\times M\to\mathbb{R}$ is a normalized Hamiltonian. We denote by $\phi^H_1$ the time-one map of the Hamiltonian flow. We refer to \cite{poltGeomBook} for details.

\begin{definition}
The \emph{displacement energy of $L$} is given by \[e_M(L)\coloneqq\inf\scomp{\lVert H\rVert}{\phi^H_1(L)\cap L=\emptyset}.\] In case no Hamiltonian displaces $L$ from itself and $e(L)=\infty$, $L$ is called \emph{non-displaceable}. We abbreviate $e(L)$ when $M$ is clear from context.
\end{definition}

From now on, we assume that $M$ is geometrically bounded in the sense of \cite{Chekanov98}. The second notion we consider is the minimal symplectic area of a pseudo-holomorphic disk with boundary on $L$, as introduced by Chekanov in \cite{Chekanov98}.

\begin{definition}
    Denote by $\mathcal{J}_\omega$ the space of $\omega$-tame almost complex structures. Set \[\invH(L,J)\coloneqq\inf\scomp{\int_Du^*\omega}{u\colon(B^2,\partial B)\to(M,L)\text{ non-const. $J$-holo.}},\] and
    \begin{equation}\label{hbarSupDefn}
    \invH(L)\coloneqq\sup_{J\in\mathcal{J}_\omega}\invH(L,J).
    \end{equation}
    Define $\invS(M)>0$ similarly with areas of pseudo-holomorphic spheres in $M$ instead of disks. Define $\hbar(L)\coloneqq\min\{\invH(L),\invS(M)\}$.
\end{definition}

As a warm-up, we consider $e$ and $\hbar$ as functions on the space of closed Lagrangians of $M$, endowed with the $C^\infty$ topology. In Section \ref{ContinuitySection}, we prove that $e$ is upper semi-continuous (can only jump down), and $\hbar$ is lower semi-continuous (can only jump up), in a suitable setup. The lower semi-continuity of $\hbar$, see Proposition \ref{hbarLSCprop}, is used in the proofs of part \ref{brendelPartOfThmItm} of Theorem \ref{eeqhTHM} and Theorem \ref{hbarlteFOOOThm} below.

Chekanov's work \cite{Chekanov98} used ideas from Lagrangian intersection Floer theory to prove the following inequality, referred to as Chekanov's bound on displacement energy in the title,
\begin{equation}\label{ChekanovsIneqEqn}
    \hbar(L)\le e(L).
\end{equation}
We observe that in many elementary examples, either
\begin{itemize}
    \item $L$ is non-displaceable, in which case \eqref{ChekanovsIneqEqn} is obvious, or
    \item $\hbar(L)=e(L)$.
\end{itemize}

Examples exhibiting one of the two options include many toric fibers and the simplest Lagrangian knot\footnote{Also called an exotic torus, a Lagrangian knot is a Lagrangian torus that is not symplectomorphic to a standard torus (the product torus in $\mathbb{C}^n$ or, more generally, a regular toric fiber in a toric symplectic manifold).}, which is the well-known monotone Chekanov torus. For instance, a product torus $T(a)\coloneqq\bigtimes_iS^1(a_i)\subset\mathbb{C}^n$ has \[\hbar(T(a))=e(T(a))=\min_ia_i\eqqcolon\underline{a}.\] Here, $S^1(b)\coloneqq\partial B^2(b)$, where $B^{2n}(\pi r^2)\subset\mathbb{C}^n$ is the closed ball with radius $r>0$.
On the other hand, the Clifford torus in complex projective space $\mathbb{C}P^n$ is non-displaceable. If $\mathbb{C}P^n$ arises by adding $\mathbb{C}P^{n-1}$ to $B^{2n}(b)$ as a hyperplane at infinity, then $\hbar$ of the Clifford torus is $\frac{b}{n+1}$.

The following proposition by Chekanov and Schlenk gives a sufficient condition for a Lagrangian torus to have $e=\hbar$.

\begin{introprop}[{\cite[Proposition 2.1]{CS16},\cite[Proposition 1.14]{brendel2023local}}]\label{prop:CS}
A Lagrangian torus $L=\varphi(T(a))\subset M^{2n}$, where $\varphi\colon B^{2n}(A)\to M$ is a Darboux chart such that \[\underline{a}+\sum_ia_i<A,\qquad\underline{a}<\sup\scomp{\invS(M,J)}{J\in\mathcal{J}_\omega,\,J|_{\operatorname{im}\varphi}=\varphi_*J_0},\] has \[\hbar(L)=e(L)=\underline{a}.\]
\end{introprop}

In light of Chekanov's bound \eqref{ChekanovsIneqEqn} and this proposition, it is natural to ask the following question.
\begin{introquestion}\label{DoesExistLQst}
    Is there a closed Lagrangian $L$ such that \[\hbar(L)<e(L)<\infty?\]
\end{introquestion}

Lagrangian tori as in Proposition \ref{prop:CS} were interpreted in \cite[Remark 1.20]{brendel2023local} as \emph{standard tori} in (not necessarily toric) symplectic manifolds\footnote{Some toric fibers, such as fibers adjacent to the central fiber of $S^2\times S^2$, may not satisfy the conditions of Proposition \ref{prop:CS}. Dylan Cant noted that neither do non-contractible circles in a surface.}. Thus, exotic Lagrangian tori are natural candidates to examine in relation to Question \ref{DoesExistLQst}.

\paragraph{The zoo of Lagrangian tori}
Both families of Lagrangians we consider are based on the Chekanov torus. Thus, we start with a brief review of the construction of the Chekanov torus in $\mathbb{C}^2$ via symplectic reduction, see also \cite{knot-survey,chekanov2010notes,joeIntroChekanovTorus}.

Recall the standard toric moment map $\mu(z_1,z_2)=(\pi\lvert z_1\rvert^2,\pi\lvert z_2\rvert^2)$ of $\mathbb{C}^2$, having $(\mathbb{R}_{\ge0})^2$ as its image. The autonomous Hamiltonian given by $\mathcal{C}\coloneqq\mu_1-\mu_2$, generates the $1$-periodic Hamiltonian flow \[t\in\mathbb{R},\quad(z_1,z_2)\mapsto(e^{2\pi it}z_1,e^{-2\pi it}z_2).\] We obtain a Hamiltonian $S^1$-action on $\mathbb{C}^2$ having the origin as its single fixed point. The level set \[\mathcal{C}^{-1}(0)=\{\lvert z_1\rvert=\lvert z_2\rvert\}\] is invariant under this $S^1$-action. After removing the origin and taking the quotient of this action, we obtain a reduced space symplectomorphic to the standard punctured plane $\mathbb{C}^*\coloneqq\mathbb{C}\setminus\{0\}$. Let $a>0$. While the product torus $T(a,a)\subset\{\lvert z_1\rvert=\lvert z_2\rvert\neq0\}$ projects to the circle $S^1(a)\subset\mathbb{C}^*$, the Chekanov torus $T^2_\text{Ch}(a)\subset\mathbb{C}^2$ is obtained by lifting a \emph{null-homotopic} curve bounding area $a$ in $\mathbb{C}^*$. The Chekanov torus is a monotone Lagrangian knot that first appeared in \cite{chekanovTorus}.

\begin{figure}[h]
\centering
\begin{tikzpicture}[x=0.3cm, y=0.3cm]
  \draw[->] (0,0) -- (9,0);
  \draw[->] (0,0) -- (0,9);

  \draw[dashed] (0,0) -- (7,7) node[above right] {$\mu(\mathcal{C}^{-1}(0))$};

  \draw[red, very thick] (1.3,1.3) -- (4,4);
\end{tikzpicture}
\caption{The {\color{red}thick segment} is the image of $T^2_\text{Ch}$ under $\mu$ in $(\mathbb{R}_{\ge0})^2$.}
\end{figure}

\paragraph{Chekanov tori in \texorpdfstring{$\mathbb{C}P^n$}{ℂPⁿ}}
Let $n\ge2$. The Chekanov torus has a straightforward generalization to higher dimensions, given for example in \cite{brendel2020reallagrangiantoriversal}. Denote the standard toric moment map by $\mu\colon\mathbb{C}^n\to(\mathbb{R}_{\ge0})^n$. Then the moment map
\begin{equation}\label{chekanovMomentMapEqn}
    \mathcal{C}=(\mu_1-\mu_2,\dotsc,\mu_1-\mu_n)\colon\mathbb{C}^n\to\mathbb{R}^{n-1}
\end{equation}
generates a $T^{n-1}$-action on $\mathbb{C}^n$. Again, performing symplectic reduction on $\mathcal{C}^{-1}(\vec{0})\setminus\{0\}$, gives a punctured plane. Lift a \emph{contractible} curve enclosing area $a>0$ in the punctured plane, to obtain a monotone Lagrangian knot $T^n_\text{Ch}(a)\subset\mathbb{C}^n$.

Chekanov tori can be embedded into closed symplectic manifolds by Darboux charts. In favorable situations, this gives rise to a family, parameterized by small enough $a$, of Lagrangian knots, only one of which is monotone. For instance, Chekanov tori in $S^2\times S^2$ are a primary subject of interest. They appear in a one-parameter family consisting of
    \begin{itemize}
        \item a non-displaceable monotone member \cite{chekanov2010notes,Oakley_2016},
        \item a one-parameter subfamily of non-displaceable non-monotone members \cite{fukaya2010toricdegenerationnondisplaceablelagrangian}, and
        \item another one-parameter subfamily of \emph{displaceable} non-monotone members \cite{lou2024lagrangiantoris2timess2}.
    \end{itemize}
    \begin{figure}[h]
\centering
     \begin{tikzpicture}
  \draw[red, line width=1pt] (4/3,2) -- (3,2);
  \draw[blue, line width=1pt] (3,2) -- (14/3,2);
  
  \fill[red] (3,2) circle (2pt);
\end{tikzpicture}\captionsetup{width=0.8\linewidth}
    \caption{A one-parameter family consisting of {\color{red}non-displaceable} members on the left and {\color{blue}displaceable} members on the right. The {\color{red}dot} corresponds to the monotone member.}
\end{figure}

For $a<\frac{\pi}{n}$, one can fit $T^n_\text{Ch}(a)$ into $B^{2n}(\pi)$. We work with the Fubini-Study symplectic form $\omega_\text{FS}$ on $\mathbb{C}P^n$, scaled such that $\int_{\mathbb{C}P^1}\omega_\text{FS}=\pi$. Let $\varphi\colon B^{2n}(\pi)\to\mathbb{C}P^n$ be the Darboux chart to the complement of a hyperplane in the complex projective space. Define \[T^n_\text{PrCh}(a)\coloneqq\varphi(T^n_\text{Ch}(a))\subset\mathbb{C}P^n.\] The Lagrangian $T^n_\text{PrCh}(a)$ is monotone only for $a=\frac{\pi}{n+1}$. We show that it is displaceable for $0<a<\frac{\pi}{n+1}$, see Proposition \ref{displaceSmallChekanovToriProp}. We call the displaceable ones \emph{small Chekanov tori}. Theorem 1.4 in \cite{brendel2020reallagrangiantoriversal} can be adapted to the non-monotone case to show that small Chekanov tori are exotic.

\paragraph{Brendel tori in \texorpdfstring{$\mathbb{C}^3$}{ℂ³}}
In \cite{brendel2023local}, Brendel defines the moment map on $\mathbb{C}^3$, for $k\ge2$, \[\nu_k\coloneqq(\mu_1+k\mu_3,\mu_2-\mu_3).\] After removing the points with a non-trivial stabilizer from the level set $\nu_k^{-1}(\pi,0)$, taking the quotient by the induced $T^2$-action \eqref{brendelT2actionEqn} produces a singular reduced space symplectomorphic to a punctured open disk.

This construction cleverly uses the one extra dimension (compared to $\mathbb{C}^2$) in the standard toric polytope of $\mathbb{C}^3$ to \emph{invert} the ray in whose $\mu$-preimage the Chekanov torus would be contained. This results in a \emph{bounded} singular reduced space, symplectomorphic to the punctured open ball $B^*\left(\frac{\pi}{k}\right)\coloneqq B^2\left(\frac{\pi}{k}\right)\setminus\{0\}\subset\mathbb{C}$. The proof is in Appendix \ref{UpsilonAppendix}. Lift a contractible curve enclosing area $\frac{\pi}{k+1}\le a<\frac{\pi}{k}$, to obtain a Lagrangian torus
\[\Upsilon_k(a)\subset\mathbb{C}^3.\]
\begin{figure}[h]
\centering
\begin{tikzpicture}[x=0.5cm,y=0.5cm,z=0.3cm,>=stealth]
\draw[->] (xyz cs:x=0) -- (xyz cs:x=10.5) node[above] {$x_1$};
\draw[->] (xyz cs:y=0) -- (xyz cs:y=7.5) node[right] {$x_3$};
\draw[->] (xyz cs:z=0) -- (xyz cs:z=7.5) node[below right] {$x_2$};

\draw[dashed] (xyz cs:x=9, y=0, z=0) -- (xyz cs:x=0, y=3, z=3) 
    node[above, xshift=7] {$\mu(\nu_k^{-1}(\pi,0))$};

\fill (xyz cs:x=9, y=0, z=0) circle (1pt);

\fill (xyz cs:x=0, y=3, z=3) circle (1pt);

\draw[red, very thick] (xyz cs:x=6.3, y=0.9, z=0.9) -- (xyz cs:x=1.5, y=2.5, z=2.5);
\end{tikzpicture}
\caption{The {\color{red}thick segment} is the image of $\Upsilon_k(a)$ under $\mu$ in $(\mathbb{R}_{\ge0})^3$.}
\end{figure}

We call $\Upsilon_k(a)$ a \emph{Brendel torus}. Brendel proved that $\Upsilon_{k_1}(a_1)$ is not symplectomorphic to $\Upsilon_{k_2}(a_2)$ for $k_1\neq k_2$ and any $\frac{\pi}{k_i+1}\le a_i<\frac{\pi}{k_i}$, yielding infinitely many exotic tori in $\mathbb{C}^3$. A detailed discussion on the relation of these tori to the work of Auroux \cite{Auroux_monTorR6} is given in Section \ref{AurouxSection}.

\begin{introthm}\label{eeqhTHM}
    Let $L$ be either one of the Lagrangians
    \begin{enumerate}
        \item a small Chekanov torus, i.e., $T^n_\textnormal{PrCh}(a)\subset\mathbb{C}P^n$ for some $0<a<\frac{\pi}{n+1}$ and $n\ge2$,
        \item\label{brendelPartOfThmItm} a monotone Brendel torus, i.e., $\Upsilon_k\left(\frac{\pi}{k+1}\right)\subset\mathbb{C}^3$ for some $k\ge2$.
    \end{enumerate}

    Then $e(L)=\hbar(L)$.
\end{introthm}

The proofs of both parts are found in Sections \ref{ChekanovSection}, \ref{brendelToriSection}, respectively. Among the methods used to calculate $e$ and $\hbar$ for these examples, the following two are especially noteworthy.

An upper bound on $e(T^n_\text{PrCh})$ and $e(\Upsilon_k)$ is achieved as follows. In both families, the tori come from lifts of curves in a $2$-dimensional reduced space. This curve is generally non-displaceable. We find that a certain optimized choice of the curve gives lifts which are compressed in their respective ambient space. This has a nice interpretation in the toric moment polytope, see Propositions \ref{displaceSmallChekanovToriProp}, \ref{dispEnergyUpsilonProp}.

The second method we mention here gives a lower bound on $\hbar(T^n_\text{PrCh})$. This is achieved by combining analysis of pseudo-holomorphic disks in a simpler case performed by Auroux with a lemma by Chekanov and Schlenk that shows that pseudo-holomorphic disks persist in large enough Darboux charts, see Proposition \ref{eAndHofChekanovProjProp}.

Brendel tori include non-monotone members for which we determine $\hbar$ (this is a corollary of $\hbar$'s lower semi-continuity), and bound $e$ from above, leaving its precise value an open problem.

\begin{introquestion}\label{dispEOfnonMonUpsQst}
    Are $e$ and $\hbar$ equal for non-monotone Brendel tori?
\end{introquestion}

To our knowledge, the only example answering Question \ref{DoesExistLQst} positively appears in two essentially equivalent forms in \cite[Theorem 1.1; Theorem 1.3; Theorem 1.4]{fukaya2011displacementpolydiskslagrangianfloer} and \cite[Example 5.6]{fukCompactToricFibers10}.

Let $S^1_\text{eq}\subset S^2(2a)$ denote the equator and let $a<A$. Set \[L_\text{FOOO}\coloneqq S^1(A)\times S^1_\text{eq}\subset\mathbb{C}\times S^2(2a).\] Note that $L_\text{FOOO}$ is a non-monotone toric fiber. A result of Lagrangian intersection Floer theory that appears in the two papers mentioned above is that $e(L_\text{FOOO})=A$ (the lower bound is the difficult part).

\begin{introthm}\label{hbarlteFOOOThm}
    $\hbar(L_\textnormal{FOOO})=a$
\end{introthm}

This is an immediate corollary of $\hbar$'s lower semi-continuity. The proof of a similar case is given in Example \ref{hLTeFntExamp}. The proof of this was, to our knowledge, previously lacking in the literature (the upper bound is the non-trivial part). Thus, \[\hbar(L_\text{FOOO})<e(L_\text{FOOO})<\infty.\]

This leads to the following variants of Question \ref{DoesExistLQst}, which remain open. 

\begin{introquestion}
    Is there a \emph{monotone} displaceable closed Lagrangian $L$ such that \[\hbar(L)<e(L)?\]
\end{introquestion}

\begin{introquestion}
    In which symplectic manifolds is there a displaceable closed Lagrangian $L$ such that \[\hbar(L)<e(L)?\]
\end{introquestion}

\paragraph{Other results}
Brendel tori $\Upsilon_k$ form infinitely many exotic tori in $\mathbb{C}^3$. The distinguishing invariant in Brendel's proof came from the \emph{displacement energy germ}. The existence of an infinite family in $\mathbb{C}^3$ of monotone Lagrangian tori, pairwise not symplectomorphic, was established earlier by Auroux \cite{Auroux_monTorR6}. The distinguishing invariant in Auroux's proof was a \emph{pseudo-holomorphic disk count}. 

In the process of computing $\hbar(\Upsilon_k)$ we determine which relative homology classes in $H_2(\mathbb{C}^3,\Upsilon_k)$ contain holomorphic disks\footnote{This was possibly already done in unpublished work by Tonkonog, Vianna, and Hicks, see \cite[Remark 1.2]{brendel2023local}.}, without assuming monotonicity. This is done using positivity of intersections. First, this gives a new proof (modulo an argument for the regularity of the standard almost complex structure) that the monotone members of the family of tori $\Upsilon_k$ are exotic, since their pseudo-holomorphic disk count depends on $k$. Second, this gives some evidence for a correspondence between the $\Upsilon_k$s and the family of Lagrangian tori constructed in \cite{Auroux_monTorR6}, by which we mean that each torus in Auroux's family is possibly Hamiltonian isotopic to some monotone $\Upsilon_k$. A further consequence is a classification up to symplectomorphism of Brendel tori. Further discussion and proofs are found in Section \ref{holoDiskUpsilonSect}.

\paragraph{Notations}
Denote the \emph{Maslov class} by $\mu_L\colon H_2(M,L)\to\mathbb{Z}$, and the \emph{symplectic area class} by $\omega_L\colon H_2(M,L)\to\mathbb{R}$. We omit the subscript $L$ from both of these ``classical'' invariants when it is clear from context.

\paragraph{Acknowledgements}
This note is based on the author's master's thesis at Tel Aviv University. I am thankful to my two supervisors, Leonid Polterovich and Sara Tukachinsky, and our collaborator Jo\'e Brendel, for their contribution and huge support. I thank Felix Schlenk for many useful remarks and the referee for various corrections and improvements. I am also grateful for the support of my family and workplace.

\section{Continuity properties of \texorpdfstring{$e$}{e} and \texorpdfstring{$\hbar$}{ℏ}}\label{ContinuitySection}

In the present section, we consider $e$ and $\hbar$ as functions on the space of closed Lagrangians of $M$, endowed with the $C^\infty$ topology.

The following fact appears in the proof of \cite[Proposition 5.11]{joeIntroChekanovTorus}.
\begin{prop}[$e$ is u.s.c]\label{eUSCProp}
    As a function on the space of compact Lagrangians, $e$ is \emph{upper semi-continuous}.
\end{prop}
\begin{proof}
    A Hamiltonian diffeomorphism that displaces a Lagrangian $L$, also displaces a Weinstein neighborhood of $L$, and in particular all the Lagrangians contained in that neighborhood.
\end{proof}

The following proposition is a form of \emph{lower semi-continuity} of $\hbar$.

\begin{prop}[$\hbar$ is l.s.c]\label{hbarLSCprop}
    Let $L\subset M$ be a closed Lagrangian and take some $C^\infty$-converging sequence of Lagrangians $L_n\to L$. Assume that $\hbar(L_n)$ converges to a number $0\le H$. Then $0<\hbar(L)\le H$.
\end{prop}

The proof is a corollary of Gromov compactness. The original version of Gromov compactness applies for fixed Lagrangian boundary condition and a converging sequence of almost complex structures. To use compactness in our case, of converging Lagrangian boundary conditions and a fixed almost complex structure, we apply a folklore trick which appears, for example, as ``Fukaya's trick'' in \cite[Section 2.1]{shelukhin2023geometrysymplecticfluxlagrangian}. See also \cite[Remark 43]{KislevShelukhin}.

First, recall the energy identity for pseudo-holomorphic curves. See \cite[Lemma 2.2.1]{MSJHolo12} for the case of curves from a compact surface. Let $J$ be $\omega$-tame. Denote the induced Riemannian metric by $g_J$, and its associated norm by $\lvert\cdot\rvert_J$. The norm of a real linear $A\colon TB\to TM$ is defined, using any section $\zeta$ of $TB$ having the standard Riemannian metric $\lvert\zeta\rvert_{TB}=1$ pointwise, by \[\lvert A\rvert_J\coloneqq\sqrt{\lvert A(\zeta)\rvert_J^2+\lvert A(i\zeta)\rvert_J^2},\] which is independent of the choice of $\zeta$. The \emph{energy} of a smooth disk $u\colon B\to M$ is then given by \[E(u)\coloneqq\frac{1}{2}\int_B\lvert du\rvert_J^2d\operatorname{vol}_B.\] The energy identity states that for $J$-holomorphic $u$, \[E(u)=\int_Bu^*\omega.\]

\begin{proof}[Proof of Proposition \ref{hbarLSCprop}]
Let $J$ be a tame almost complex structure on $M$. By Gromov compactness, the infimum in the definition of $\hbar(L_n,J)$ is realized by a $J$-holomorphic curve. If these curves are spheres for arbitrarily large $n$, then $\lim_n\hbar(L_n,J)\ge\hbar(L,J)$, because such spheres are accounted for in $\hbar(L,J)$. Otherwise, there exist $J$-holomorphic disks $u_n\colon(B,\partial B)\to(M,L_n)$, having $\omega(u_n)=\hbar(L_n,J)>0$. Use Fukaya's trick, as follows. Let $f_n$ be diffeomorphisms taking $L_n$ to $L$ inside a Weinstein neighborhood $U$ of $L$, having $f_n\xrightarrow[n\to\infty]{C^\infty}\IdMap_M$. Set $J_n\coloneqq(f_n)_*J$. Then $J_n\xrightarrow[n\to\infty]{C^\infty}J$ and for $\xi,\eta\in TM$,
\begin{align}\label{eqn:MetricsConverge}
    \lvert g_{J_n}(\xi,\eta)-g_J(\xi,\eta)\rvert&=\frac{1}{2}\lvert\omega(\xi,J_n\eta)+\omega(\eta,J_n\xi)-\omega(\xi,J\eta)-\omega(\eta,J\xi)\rvert \nonumber \\ &=\frac{1}{2}\lvert\omega(\xi,J_n\eta-J\eta)+\omega(\eta,J_n\xi-J\xi)\rvert \nonumber \\ &\le\max_{\rvert a\lvert_J=\rvert b\lvert_J=1}\lvert\omega(a,b)\rvert\cdot\max_{\overline{U}}\lVert(J_n-J)\rVert_\text{op}\cdot\rvert\xi\lvert_J\cdot\rvert\eta\lvert_J.
\end{align}

The $J_n$-holomorphic disk $v_n\coloneqq f_n\circ u_n$ has boundary on $L$. Since $\omega$-tameness is an open condition, $J_n$ is $\omega$-tame for large enough $n$. Thus,
\[
\lvert\omega(v_n)-\omega(u_n)\rvert=\left\lvert\frac{1}{2}\int_B\lvert df_n\circ du_n\rvert_{J_n}^2d\operatorname{vol}_B-\frac{1}{2}\int_B\lvert du_n\rvert_J^2d\operatorname{vol}_B\right\rvert\xrightarrow[n\to\infty]{}0,\]
by the dominated convergence theorem and the facts that $df_n\xrightarrow[n\to\infty]{}\IdMap_{TM}$ in the compact Weinstein neighborhood $\overline{U}$, and $df_n=\IdMap_{TM}$ outside $\overline{U}$, and by the convergence of the metrics \eqref{eqn:MetricsConverge}.

Consequently, $\omega(v_n)$ is bounded uniformly in $n$. By Gromov compactness, there exists a subsequence, still denoted by $v_n$, that converges in the Gromov sense to some $J$-holomorphic disk $v$ with boundary on $L$, and there also exists a finite collection of $J$-holomorphic curves $g_k\colon(B^2,\partial B)\to(M,L)$ and $h_k\colon S^2\to M$ such that
\[\lim\omega(v_n)=\omega(v)+\sum_k\omega(g_k)+\sum_k\omega(h_k).\]

It cannot be that $\lim\omega(v_n)=0$, which is when the limit is a point. In this case, $v_n$ $C^\infty$-converges to a constant map. So for large enough $n$, its image is contained in a $(\dim M)$-ball intersecting $L$ in a $(\dim L)$-ball, so $[(v_n)_*(B)]=0\in H_2(M,L)$. On the other hand, $\omega(v_n)>0$ since $v_n$ is non-constant $J_n$-holomorphic.

Thus, either $v$ or some $g_k$, or some $h_k$ is $J$-holomorphic, non-constant and its area $A$ has $\hbar(L,J)\le A\le\lim_n\omega(v_n)\le H$. This works for arbitrary tame $J$, so the proof is complete.
\end{proof}

This is reminiscent of the work \cite{shelukhin2023geometrysymplecticfluxlagrangian} studying the invariant $\Psi(L)>0$, which is defined as the minimal symplectic area of pseudo-holomorphic disks with boundary on $L$, satisfying a certain condition coming from the Fukaya-$A_\infty$ algebra. Since $\invH(L)$ is the minimum symplectic area of \emph{all} pseudo-holomorphic disks with boundary on $L$, $\invH(L)\le\Psi(L)$. Theorem A in \cite{shelukhin2023geometrysymplecticfluxlagrangian} states that $\Psi$ is continuous when evaluated on Lagrangians isotopies.

One of the main motivations of the present project is to seek examples of \emph{displaceable} Lagrangians where $\hbar\neq e$. The following $4$-dimensional example was found by \cite{fukaya2011displacementpolydiskslagrangianfloer}, \cite[Example 5.6]{fukCompactToricFibers10}. 

{\setlength{\topsep}{0pt}\setlength{\belowdisplayskip}{0pt}
\begin{example}\label{hLTeFntExamp}
Let $S^2(a)\subset\mathbb{R}^3$ denote the unit sphere with the standard area form scaled to have total area $a$. Take the product of two spheres with areas $2a<4a$. This is a non-monotone closed toric symplectic manifold, $M=S^2(2a)\times S^2(4a)\subset(\mathbb{R}^3)^2$, with moment polytope $\Delta=[-a,a]\times[-2a,2a]$. The moment map is
$\mu(x_1,y_1,z_1,x_2,y_2,z_2)=(az_1,2az_2)$.
Denote by $T(p)\subset M$ the fiber over $p\in\Delta$. Fix $0<\epsilon<a$ and let \[{\color{blue}L\coloneqq T(0,a-\epsilon)}\subset M,\]
i.e., the equator in the smaller sphere times a curve bounding area $a+\epsilon$ in the bigger sphere. Then $L$ is a non-monotone Lagrangian torus with $\hbar(L)=a<a+\epsilon=e(L)<\infty$. 
{\setlength{\textfloatsep}{6pt}
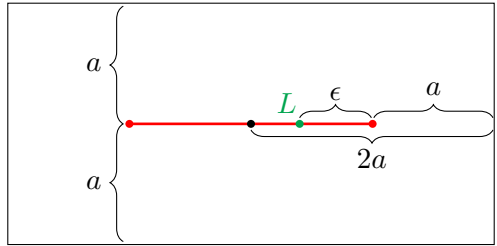
\begin{figure}[h]
\centering
     \begin{tikzpicture}[x=0.8cm,y=0.8cm]
     \def\eps{0.1}
  \draw (0,0) -- (8,0) -- (8,4) -- (0,4) -- cycle;

  \draw[red, line width=1pt] (2.05,2) -- (5.95,2);

  \draw[decorate, decoration={brace,amplitude=5pt,mirror}] (2-\eps,3.95) -- (2-\eps,2) node[midway,xshift=-11pt] {\small$a$};
  \draw[decorate, decoration={brace,amplitude=5pt,mirror}] (2-\eps,2) -- (2-\eps,0.05) node[midway,xshift=-11pt] {\small$a$};
  
  \draw[red, thick] (2,2) circle [radius=1.5pt];
  \draw[decorate, decoration={brace,amplitude=5pt}] (6.02,2+\eps) -- (7.98,2+\eps) node[midway,yshift=11pt] {\small$a$};
  \draw[decorate, decoration={brace,amplitude=5pt,mirror}] (4,2-\eps) -- (7.98,2-\eps) node[midway,yshift=-11pt] {\small$2a$};
  \draw[red, thick] (6,2) circle [radius=1.5pt];
  \draw[decorate, decoration={brace,amplitude=5pt}] (4.8,2+\eps) -- (5.98,2+\eps) node[midway,yshift=9pt] {\small$\epsilon$};
  \fill[blue] (4.8,2) circle (1.5pt) node[above left, xshift=3pt]{\small$L$};
  \fill (4,2) circle (1.5pt);
\end{tikzpicture}\captionsetup{width=0.8\linewidth} 
    \caption{Proposition \ref{eUSCProp} in action. The {\color{red}open segment} is the discontinuity set of $p\in\Delta\mapsto e(T(p))$. The middle fiber is non-displaceable. In contrast, $p\in\Delta\mapsto\hbar(T(p))$ is continuous.}
    \label{fig:CounterExampleRectangle}
\end{figure}
}
\end{example}
}
\begin{proof}
    The second factor of $L$ can be displaced by a Hamiltonian isotopy of Hofer energy, arbitrarily close to the smaller area enclosed by this factor, which is $a+\epsilon$. This isotopy displaces $L$. Thus, this is an upper bound on $e(L)$. The lower bound comes from \cite[Example 5.6]{fukCompactToricFibers10}. Thus, $e(L)=a+\epsilon$.

    On the other hand, we apply Proposition \ref{hbarLSCprop} to show $\hbar(L)=a$. Indeed, let $N=T(u,v)$ be a fiber outside the thick segment in Figure \ref{fig:CounterExampleRectangle}, i.e., $u\neq0$ or $\lvert v\rvert>a$. Then $e(N)=\hbar(N)=\min\{a-\lvert u\rvert,2a-\lvert v\rvert\}$ is the distance of the fiber from the boundary of the rectangle. Indeed, by Chekanov's bound \eqref{ChekanovsIneqEqn} we have $e\ge\invH$. There are no contributions from pseudo-holomorphic spheres, because $\omega(H_2(M))=\mathbb{Z}[2a]$. That is, all non-constant spheres have areas bigger than that of the smallest disks. The upper bound on $e(N)$ is then given by a Hamiltonian isotopy which displaces the ``smaller'' circle of the product torus $N$ and leaves the other circle fixed. The lower bound on $\hbar(N)$ is given by the area of the smallest $J_0$-holomorphic disk on $N$ which bounds the ``smaller'' circle (and is constant on the bigger circle). This also gives a lower bound on $\hbar(L)$.

    To apply Proposition \ref{hbarLSCprop}, take the convergent sequence of Lagrangians $T\left(\frac{1}{n},a-\epsilon\right)\xrightarrow[n\to\infty]{C^\infty}L$, having $\hbar(T\left(\frac{1}{n},a-\epsilon\right))=a-\frac{1}{n}\xrightarrow[n\to\infty]{}a$. 
    Thus $\hbar(L)=a$.
\end{proof}

\section{Chekanov tori}\label{ChekanovSection}

Let $n\ge2$. We start with a brief review of the construction of the Chekanov torus in $\mathbb{C}^n$ via symplectic reduction. We refer to \cite{knot-survey,chekanov2010notes} and \cite[Section 6.2]{joeIntroChekanovTorus} for details.

\paragraph{In \texorpdfstring{$\mathbb{C}^n$}{ℂⁿ}}
Recall the standard toric moment map $\mu(z_1,\dotsc,z_n)=\left(\pi\lvert z_1\rvert^2,\dotsc,\pi\lvert z_n\rvert^2\right)$ of $\mathbb{C}^n$, having $(\mathbb{R}_{\ge0})^n$ as its image. The moment map \eqref{chekanovMomentMapEqn} generates the $T^{n-1}$-action \[\theta\in\scomp{\theta\in T^n}{\sum_j\theta_j=1}\cong T^{n-1},\quad\theta\cdot(z_1,\dotsc,z_n)\mapsto\left(e^{2\pi i\theta_1}z_1,\dotsc,e^{2\pi i\theta_n}z_n\right).\] This is a Hamiltonian $T^{n-1}$-action on $\mathbb{C}^n$ having the origin as its single fixed point. The level set \[\mathcal{C}^{-1}(\vec{0})=\scomp{(z_1,\dotsc,z_n)\in\mathbb{C}^n}{\lvert z_1\rvert=\lvert z_j\rvert\,\forall j},\] is invariant under this action. After removing the origin from this level set, and taking the $T^{n-1}$-quotient, we obtain a reduced space symplectomorphic to the standard punctured plane $\mathbb{C}^*\coloneqq\mathbb{C}\setminus\{0\}$. Let $a>0$. While the monotone product torus $T(a,\dotsc,a)\subset\mathcal{C}^{-1}(\vec{0})$ projects to the circle $S^1(a)\subset\mathbb{C}^*$, the Chekanov torus $T^n_\text{Ch}(a)\subset\mathbb{C}^n$ is obtained by lifting a \emph{null-homotopic} loop $\Gamma\subset\mathbb{C}^*$ bounding area $a$,
\begin{equation}\label{highDimChekanovTorusExprEqn}
    \resizebox{.9\hsize}{!}{$T^n_\text{Ch}(a)=\scomp{\left(\lvert z\rvert e^{2\pi i\theta_1},\dotsc,\lvert z\rvert e^{2\pi i\theta_{n-1}},e^{2\pi i\theta_n}z\right)}{z\in\Gamma,\,\theta\in T^n,\,\sum_j\theta_j=1}$.}
\end{equation}

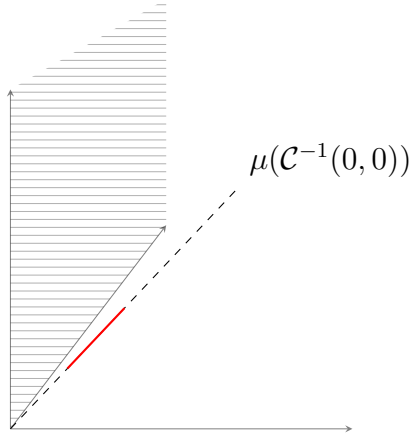
\begin{figure}[h]
\centering
    \begin{tikzpicture}[x=0.35cm,y=0.35cm,z=0.21cm,>=stealth]
\draw[->] (xyz cs:x=0) -- (xyz cs:x=9) node[above] {$x_1$};
\draw[->] (xyz cs:y=0) -- (xyz cs:y=9) node[right] {$x_3$};
\draw[->] (xyz cs:z=0) -- (xyz cs:x=-1.3,z=9) node[above] {$x_2$};

\draw[dashed] (xyz cs:x=0, y=0, z=0) -- (xyz cs:x=3.6, y=4, z=4) 
    node[above right] {$\mu(\mathcal{C}^{-1}(0,0))$};

\draw[red, very thick] (xyz cs:x=0.9, y=1, z=1) -- (xyz cs:x=1.8, y=2, z=2);
\end{tikzpicture}
\caption{The {\color{red}thick segment} is the image of $T^3_\text{Ch}$ under $\mu$ in $(\mathbb{R}_{\ge0})^3$.}
\end{figure}

The Chekanov torus is a monotone Lagrangian knot. It is simple to describe the holomorphic disks with boundary on $T^n_\text{Ch}(a)$.
\begin{prop}[{\cite[Proposition 4.2.C]{knot-survey}} and {\cite[Example 3.3.1]{auroux2009speciallagrangianfibrationswallcrossing}}]\label{chekanovMultiDimDisksProp}
    There is exactly one class $\alpha\in H_2(\mathbb{C}^n,T^n_\textnormal{Ch}(a))$ with $\mu(\alpha)=2$ and a holomorphic representative. 
\end{prop}

The class $\alpha$ has area $a$ and minimal positive Maslov index, so by monotonicity,
\begin{equation}\label{multiDimChekanovstdJinvHEqn}
    \invH(T^n_\text{Ch}(a),J_0)=a.
\end{equation}

\paragraph{In \texorpdfstring{$\mathbb{C}P^n$}{ℂPⁿ}}
Set $B^*\left(\frac{\pi}{n}\right)\coloneqq B^2\left(\frac{\pi}{n}\right)\setminus\{0\}\subset\mathbb{C}$. For $0<a<\frac{\pi}{n}$, there exists a null-homotopic curve $\Gamma\subset B^*\left(\frac{\pi}{n}\right)$ of area $a$. By \eqref{highDimChekanovTorusExprEqn}, $T^n_\text{Ch}(a)\subset B^{2n}(\pi)$ for such $\Gamma$. Denote a standard Darboux chart onto the complement $\mathcal{U}\subset\mathbb{C}P^n$ of a hyperplane by $\varphi\colon B^{2n}(\pi)\to\mathcal{U}$. Here the Fubini-Study form $\omega_\text{FS}$ has $\int_{\mathbb{C}P^1}\omega_\text{FS}=\pi$. Thus, we define the Chekanov torus in $\mathbb{C}P^n$ by
\begin{equation}\label{chekanovTorusProjSpaceChartEqn}
    T^n_\text{PrCh}(a)\coloneqq\varphi\left(T^n_\text{Ch}(a)\right).
\end{equation}

It is only monotone for $a=\frac{\pi}{n+1}$. It is believed that in the case $a\ge\frac{\pi}{n+1}$, $T^n_\text{PrCh}(a)$ is non-displaceable, see \cite{WuLagTorusInCP2}, \cite[Conjecture 1.1]{lowAreaFloerTonkonogVianna}\footnote{Thank Peleg Bar Lev for pointing out \cite{lowAreaFloerTonkonogVianna}.}. In this note, we study the \emph{small} case $a<\frac{\pi}{n+1}$, cf. \cite{lou2024lagrangiantoris2timess2}.

In order to displace $T^n_\text{PrCh}(a)\subset\mathbb{C}P^n$, we first optimize the position inside the reduced space of the curve $\Gamma\subset\mathbb{C}^*$ whose lift yields $T^n_\text{Ch}\subset\mathbb{C}^n$. This gives a $T^n_\text{Ch}$ which is not spread out in the ambient $\mathbb{C}^n$ space. Explicitly, let $\frac{\pi}{n+1}-a>\epsilon>0$. Choose $\Gamma\subset B^*\left(\frac{\pi}{n}\right)$ to be contained in $B(a+\epsilon)$, see similar situation in Figure \ref{fig:optimizedChekanovCurveForBrendelTorus}. Then by \eqref{highDimChekanovTorusExprEqn}, the Chekanov torus is contained in the polydisk \begin{equation}\label{chekanovTorusPolydiskEqn}
    T_\text{Ch}^n(a)\subset B^2(a+\epsilon)^n\subset B^{2n}(\pi).
\end{equation}

\begin{figure}[h]
    \centering
\begin{tikzpicture}[x=0.8cm,y=0.8cm]    
    \draw (0,0) -- (4,0) -- (0,4) -- cycle;

    \node at (4/3,4/3) {$\times$};
    \draw[blue, very thick] (0.1,0.1) -- (1.2,1.2);
    \draw[gray, dashed] (0,1.25) -- (1.25,1.25) -- (1.25,0);
\end{tikzpicture}
    \caption{A bi-disk containing {\color{blue}$T_\text{PrCh}^2(a)$} in $\mathbb{C}P^2$ for $a<\frac{\pi}{3}$.}
\end{figure}

We now write an explicit Hamiltonian flow that displaces small Chekanov tori in $\mathbb{C}P^2$. The generalization to $\mathbb{C}P^n$ is straightforward. The following is an autonomous Hamiltonian $\mathbb{C}P^2\to\mathbb{R}$ which I learned from Lily Reisch. The Hamiltonian given by \[[w_0\colon w_1\colon w_2]\mapsto\frac{\pi\lvert w_1-w_2\rvert^2}{2\sum_j\lvert w_j\rvert^2}\] is clearly well-defined, and generates the flow, for $t\in\mathbb{R}$,
\[[w_0\colon w_1\colon w_2]\mapsto\left[w_0\colon\frac{w_1+w_2+e^{\pi it}(w_1-w_2)}{2}\colon\frac{w_2+w_1+e^{\pi it}(w_2-w_1)}{2}\right]\] whose time-$1$ map is $\psi$ \begin{equation}\label{permSympEqn}
    [w_0\colon w_1\colon w_2]\mapsto[w_0\colon w_2\colon w_1].
\end{equation}

\begin{figure}[h]
    \centering
\begin{tikzpicture}[x=0.7cm,y=0.7cm]    
    \draw (0,0) -- (4,0) -- (0,4) -- cycle;

    \node at (4/3,4/3) {$\times$};
    \draw[blue, very thick] (1.2,1.2) -- (0.1,0.1) node[below left]{$T_\text{PrCh}^2(a)$};
    \draw[red, very thick] (1.5,1.25) -- (3.8,0.1) node[below right]{$\psi\left(T_\text{PrCh}^2(a)\right)$};
\end{tikzpicture}
    \caption{The image of $T_\text{PrCh}^2(a)$ in $\mathbb{C}P^2$ under \eqref{permSympEqn} for $a<\frac{\pi}{3}$.}
\end{figure}

This explanation appears in \cite[Proposition 3.4]{lowAreaFloerTonkonogVianna}. However, it is not easy to see the most efficient displacement from the above approach.

\begin{prop}\label{displaceSmallChekanovToriProp}
    For $a<\frac{\pi}{n+1}$, we have $e\left(T_\textnormal{PrCh}^n(a)\right)\le a$.
\end{prop}
\begin{proof}
Let $\frac{\pi}{n+1}-a>\epsilon>0$ and $\pi-(n+1)(a+\epsilon)>\delta>0$. By Lemma \ref{lem:displaceDisk},
\[e_{B^2(2a+2\epsilon+\delta)}(B^2(a+\epsilon))\le a+\epsilon.\] That is, there exists a Hamiltonian isotopy $\phi_t\colon\mathbb{C}\to\mathbb{C}$, supported in $B^2(2a+2\epsilon+\delta)$, which displaces $B^2(a+\epsilon)$ from itself, and has energy at most $a+\epsilon+\delta$. The isotopy $\Phi_t\colon B^{2n}(\pi)\to B^{2n}(\pi)$ given by $(z_1,\dotsc,z_n)\mapsto(\phi_t(z_1),z_2,\dotsc,z_n)$ is well-defined by the choice of $\delta$. It clearly displaces $B^2(a+\epsilon)^n$ from itself. Cut-off $\Phi_t$ outside a neighborhood of the trajectory $\bigcup_t\Phi_t(B^2(a+\epsilon)^n)\subset B^{2n}(\pi)$, small enough to make its energy at most $a+\epsilon+2\delta$. Taking $\delta\searrow0$, \[e_{B^{2n}(\pi)}(B^2(a+\epsilon)^n)\le a+\epsilon.\] This compactly supported isotopy in $\mathcal{U}=\C P^n\setminus\C P^{n-1}$ can be extended to $\C P^n$, so
\begin{align*}
    e_{\mathbb{C}P^n}(T_\text{PrCh}^n(a))&\le e_\mathcal{U}(T_\text{PrCh}^n(a)) \\ &=e_{B^{2n}(\pi)}(T_\text{Ch}^n(a)) &\text{by }\eqref{chekanovTorusProjSpaceChartEqn} \\ &\le e_{B^{2n}(\pi)}(B^2(a+\epsilon)^n) &\text{by }\eqref{chekanovTorusPolydiskEqn} \\ &\le a+\epsilon.
\end{align*}
This is true for arbitrarily small $\epsilon>0$, so we are done.
\end{proof}

A useful fact about $\invH$ is that it respects symplectic embeddings in the following sense.
\begin{prop}[{\cite[Lemma 2.2]{CS16}}]\label{csLagEmbeddinvHLem}
    Let $L\subset B^{2n}(a)\subset\mathbb{C}^n$ be a closed Lagrangian and $\varphi\colon B^{2n}(A)\to M$ be a Darboux chart with $A>a$. Then \[\invH(\varphi(L))\ge\min\{\invH(L,J_0),\,A-a\}.\]
\end{prop}
\begin{prop}\label{eAndHofChekanovProjProp}
    For $a<\frac{\pi}{n+1}$, we have \[e\left(T_\textnormal{PrCh}^n(a)\right)=\hbar\left(T_\textnormal{PrCh}^n(a)\right)=a.\]
\end{prop}
\begin{proof}
By Proposition \ref{displaceSmallChekanovToriProp} and Chekanov's bound \eqref{ChekanovsIneqEqn} it is enough to show $\hbar\left(T_\text{PrCh}^n(a)\right)\ge a$.
The embedding \eqref{chekanovTorusPolydiskEqn} can be improved to \[T_\text{Ch}^n(a)\subset B^2(a+\epsilon)^n\subset B^{2n}\left(\frac{\pi n}{n+1}\right).\] Now by Equation \eqref{chekanovTorusProjSpaceChartEqn} and Proposition \ref{csLagEmbeddinvHLem} we have the inequality \begin{align*}
    \invH(T_\text{PrCh}^n(a))&\ge\min\left\{\invH(T_\text{Ch}^n(a),J_0),\,\pi-\frac{\pi n}{n+1}\right\} \\ &=\min\left\{a,\,\frac{\pi}{n+1}\right\} &\text{by }\eqref{multiDimChekanovstdJinvHEqn} \\ &=a.
\end{align*}

Finally, note that $\omega(H_2(\mathbb{C}P^n))=\omega(\mathbb{Z}[\mathbb{C}P^1])=\mathbb{Z}[\pi]$ which means that no spheres contribute to Chekanov's bound \eqref{ChekanovsIneqEqn}.
\end{proof}

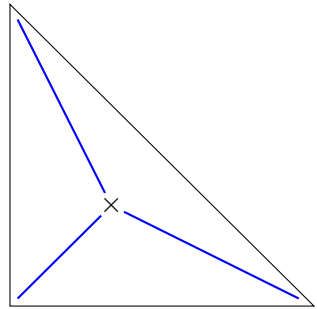
\begin{question}
    Do more than $n+1$ copies of $T_\textnormal{PrCh}^n(a)$ fit in $\mathbb{C}P^n$ for all $a<\frac{\pi}{n+1}$?
\begin{figure}[h]
    \centering
\begin{tikzpicture}[x=0.8cm,y=0.8cm]
    \draw (0,0) -- (4,0) -- (0,4) -- cycle;

    \node at (4/3,4/3) {$\times$};
    \draw[blue, very thick] (0.1,0.1) -- (1.2,1.2);
    \draw[blue, very thick] (3.8,0.1) -- (1.5,1.25);
    \draw[blue, very thick] (0.1,3.8) -- (1.25,1.5);
\end{tikzpicture}
    \caption{Do more than $3$ copies of {\color{blue}$T_\text{PrCh}^2(a)$} fit in $\mathbb{C}P^2$ for $\frac{\pi}{3}\times0.99<a<\frac{\pi}{3}$? }
\end{figure}
\end{question}

\section{Brendel tori in \texorpdfstring{$\mathbb{C}^3$}{ℂ³}}\label{brendelToriSection}

In \cite{brendel2023local}, Brendel defines the moment map on $\mathbb{C}^3$, for $k\ge2$, \[\nu_k\coloneqq(\mu_1+k\mu_3,\mu_2-\mu_3).\] 
It generates a Hamiltonian action by a $2$-dimensional torus: 
\begin{equation}\label{brendelT2actionEqn}
    (\theta_1,\theta_2)\cdot(z_1,z_2,z_3)=\left(e^{2\pi i\theta_1}z_1,e^{2\pi i\theta_2}z_2,e^{2\pi i(k\theta_1-\theta_2)}z_3\right).
\end{equation}

After removing points with a non-trivial stabilizer from the level set $\nu_k^{-1}(\pi,0)$, we remain with the set given in polar coordinates on $\mathbb{C}^3$, \[Z_k\coloneqq\{r_1\neq0\neq r_2=r_3,1=r_1^2+kr_3^2\}\subset\C^3.\] Define the smooth map \begin{align}
\begin{split}\label{eqn:qProjectionFromLevelSets}
    q\colon Z_k&\to\mathbb{C}^* \\ (r_1 e^{2\pi i\theta_1},r_2 e^{2\pi i\theta_2},r_3 e^{2\pi i\theta_3})&\mapsto r_3 e^{2\pi i(\theta_2+\theta_3-k\theta_1)},
\end{split}
\end{align} which is clearly constant on the orbits of the $T^2$-action \eqref{brendelT2actionEqn}.

Taking the quotient of $Z_k$ by the induced $T^2$-action \eqref{brendelT2actionEqn} produces a singular reduced space symplectomorphic (via the map induced by $q$) to the punctured open disk $B^*\left(\frac{\pi}{k}\right)\coloneqq B^2\left(\frac{\pi}{k}\right)\setminus\{0\}\subset\mathbb{C}$, see Appendix \ref{UpsilonAppendix}. Lift a contractible curve $\Gamma\subset B^*\left(\frac{\pi}{k}\right)$, enclosing area $\frac{\pi}{k+1}\le a<\frac{\pi}{k}$, to obtain a Lagrangian torus in $\mathbb{C}^3$ given by
\begin{align}
\Upsilon_k(a)&=q^{-1}(\Gamma) \nonumber \\ &=\scomp{\left(\sqrt{1-k\lvert z\rvert^2}e^{2\pi i\theta_1},\lvert z\rvert e^{2\pi i\theta_2},e^{2\pi i(k\theta_1-\theta_2)}z\right)}{\theta_i\in S^1,\,z\in\Gamma}. \label{UpsilonExprEqn}
\end{align}
\emph{Warning:} Our convention for the parameter $a$ differs from that of \cite{brendel2023local}.

Brendel proved that $\Upsilon_k$ is not symplectomorphic to $\Upsilon_{k'}$ for $k\neq k'$, yielding infinitely many exotic tori in $\mathbb{C}^3$. The existence of an infinite family in $\mathbb{C}^3$ of monotone Lagrangian tori, pairwise not symplectomorphic, was established earlier by Auroux \cite{Auroux_monTorR6}. More on the connection between the two works can be found in Section \ref{AurouxSection}.

\subsection{\texorpdfstring{$e=\hbar$}{e=ℏ} in the monotone case}

Hamiltonian isotopies of the curve $\Gamma\subset B^*\left(\frac{\pi}{k}\right)$ in the reduced space induce Hamiltonian isotopies of $\Upsilon_k\subset\mathbb{C}^3$. However, $\Gamma\subset B^*\left(\frac{\pi}{k}\right)$ is non-displaceable since $a\ge\frac{\pi}{k+1}$ is greater than half of the area of the entire punctured disk. Thus we apply the idea used to bound the displacement energy of $T^n_\text{PrCh}\subset\mathbb{C}P^n$ in Proposition \ref{displaceSmallChekanovToriProp}.
\begin{prop}\label{dispEnergyUpsilonProp}
     $e(\Upsilon_k)\le a$
\end{prop}
\begin{proof}
Fix $\epsilon>0$ with $\epsilon<\frac{\pi}{k}-a$. Choose $\Gamma\subset B^*\left(\frac{\pi}{k}\right)$ to be contained in $B(a+\epsilon)$, as in Figure \ref{fig:optimizedChekanovCurveForBrendelTorus}. By the explicit expression \eqref{UpsilonExprEqn}, we have $\Upsilon_k\subset\mathbb{C}^2\times B(a+\epsilon)$. We can displace $B(a+\epsilon)$ with energy $a+2\epsilon$, so we can also displace $\Upsilon_k(a)$ with energy $a+2\epsilon$. This is true for arbitrarily small $\epsilon$, so $e(\Upsilon_k)\le a$.
\end{proof}
\begin{figure}[H]
\centering
\begin{tikzpicture}[x=0.71cm,y=0.71cm]
\def\gap{0.3}
\def\outerradius{3}
\def\smallepsilon{0.2}
\def\bigradius{2.3}
\def\littleradius{0.15}

\draw [help lines,->] (-1.15*\outerradius, 0) -- (1.15*\outerradius,0);
\draw [help lines,->] (0, -1.15*\outerradius) -- (0, 1.15*\outerradius);
  \draw[blue,line width=1pt]
  let
     \n1 = {asin(\gap/2/\bigradius)},
     \n2 = {asin(\gap/2/\littleradius)}
  in (5+\n1:\bigradius) arc (5+\n1:340-\n1:\bigradius)
  -- (\bigradius-0.02,-\littleradius-0.13) arc (0:90:0.15)
  -- (-\n2:\littleradius) arc (-\n2:-360+\n2:\littleradius)
  -- (\bigradius-0.23,\littleradius) arc (270:360:0.2);

\draw[line width=0.5pt]
  (0:\bigradius+\smallepsilon) arc (0:360:\bigradius+\smallepsilon)
  -- cycle;

\draw[gray, line width=0.5pt]
  (0:\outerradius) arc (0:360:\outerradius)
  -- cycle;

\node at (-1.53,1.2) {\color{blue}$\Gamma$};
\node at (2.82,-1.7) {\small$B(a+\epsilon)$};
\node at (-2.3,2.82) {\color{gray}$B^*\left(\frac{\pi}{k}\right)$};
\node at (0,0) {\tiny$\times$};
\end{tikzpicture}
    \caption{The $\epsilon$-optimized curve {\color{blue}$\Gamma$} in the reduced space {\color{gray}$B^*\left(\frac{\pi}{k}\right)$}.}
    \label{fig:optimizedChekanovCurveForBrendelTorus}
\end{figure}

The idea of this proof is to optimize the choice of the curve that we lift from the reduced space, so that $\Upsilon_k$ would be as narrow as possible as a subset of $\mathbb{C}^3$ (regardless of it coming from a symplectic reduction). This has a nice geometric interpretation in the moment polytope of $\mathbb{C}^3$. The image $\mu(\Upsilon_k)$ in the moment polytope is a segment inside the line segment \[\ell_k\coloneqq\mu(\nu_k^{-1}(\pi,0))\subset(\mathbb{R}_{\ge0})^3,\] on the preimage of which symplectic reduction is preformed. Translating this segment towards the end of $\ell_k$ which is on the $x_1$ axis, corresponds to the described optimization, see Figure \ref{fig:optimizedChoiceOfCurveInMomentPolytopeBrendel}.
\begin{figure}[h]
    \centering
    \begin{minipage}{0.45\textwidth}
        \centering
        \begin{tikzpicture}[x=0.5cm,y=0.5cm,z=0.3cm,>=stealth]
\draw[->] (xyz cs:x=0) -- (xyz cs:x=10.5) node[above] {$x_1$};
\draw[->] (xyz cs:y=0) -- (xyz cs:y=7.5) node[right] {$x_3$};
\draw[->] (xyz cs:z=0) -- (xyz cs:z=7.5) node[above] {$x_2$};
\draw[dashed] (xyz cs:x=9, y=0, z=0) -- (xyz cs:x=0, y=3, z=3) 
    node[above right] {$\ell_k$};
\fill (xyz cs:x=9, y=0, z=0) circle (1pt);
\fill (xyz cs:x=0, y=3, z=3) circle (1pt);
\draw[red, very thick] (xyz cs:x=6.3, y=0.9, z=0.9) -- (xyz cs:x=1.5, y=2.5, z=2.5);
\end{tikzpicture}
        \caption*{Any {\color{red}segment} confined to $\ell_k$.}
    \end{minipage}
    \hfill
    \begin{minipage}{0.45\textwidth}
        \centering
\begin{tikzpicture}[x=0.5cm,y=0.5cm,z=0.3cm,>=stealth]
\draw[->] (xyz cs:x=0) -- (xyz cs:x=10.5) node[above] {$x_1$};
\draw[->] (xyz cs:y=0) -- (xyz cs:y=7.5) node[right] {$x_3$};
\draw[->] (xyz cs:z=0) -- (xyz cs:z=7.5) node[above] {$x_2$};
\draw[dashed] (xyz cs:x=9, y=0, z=0) -- (xyz cs:x=0, y=3, z=3) 
    node[above right] {$\ell_k$};    
\fill (xyz cs:x=9, y=0, z=0) circle (1pt);
\fill (xyz cs:x=0, y=3, z=3) circle (1pt);
\draw[red, very thick] (xyz cs:x=8.7, y=0.1, z=0.1) -- (xyz cs:x=4.2, y=1.6, z=1.6);
\end{tikzpicture}
        \caption*{An $\epsilon$-optimized {\color{red}segment} on $\ell_k$.}
    \end{minipage}
    \caption{The image in the moment polytope of two choices of the curve that we lift from the reduced space to obtain $\Upsilon_k$.}
    \label{fig:optimizedChoiceOfCurveInMomentPolytopeBrendel}
\end{figure}

The curve $\Gamma$ cuts $\C$ into two connected components, one of which is a closed disk $D$ with $0\notin D$ and $\partial D=\Gamma$.

\begin{prop}\label{BrendelRelativeHomologyProp}
The following homology classes span $H_2(\mathbb{C}^3,\Upsilon_k)$: \[w\in D\xmapsto{\alpha}\left(\sqrt{1-k\lvert w\rvert^2},\lvert w\rvert,w\right),\] and \[z\in\overline{B}(\pi)\xmapsto{\beta_1}\left(\sqrt{1-k\lvert w_0\rvert^2}z,\lvert w_0\rvert,w_0z^k\right),\] and \[z\in\overline{B}(\pi)\xmapsto{\beta_2}\left(\sqrt{1-k\lvert w_0\rvert^2},\lvert w_0\rvert z,w_0\bar{z}\right)\] for fixed $w_0\in\Gamma$. Additionally,
\[
\begin{tabular}{|c||c|c|} 
 \hline
 \textnormal{class} & $\omega$ & $\mu$ \\
 \hline
 $\alpha$ & $a$ & $2$ \\
 $\beta_1$ & $\pi$ & $2k+2$ \\
 $\beta_2$ & $0$ & $0$ \\
 \hline
\end{tabular}
\]
\end{prop}

The proof is a straightforward computation. In particular, $\Upsilon_k$ is monotone if there exists $\kappa>0$ s.t \[a=\omega(\alpha)=\kappa\mu(\alpha)=2\kappa\] and \[\pi=\omega(\beta_1)=\kappa\mu(\beta_1)=2\kappa(k+1),\] which means $a=\frac{\pi}{k+1}$.

\begin{corollary}
    $e\left(\Upsilon_k\left(\frac{\pi}{k+1}\right)\right)=\hbar\left(\Upsilon_k\left(\frac{\pi}{k+1}\right)\right)=\frac{\pi}{k+1}$
\end{corollary}
\begin{proof}
Chekanov's bound \eqref{ChekanovsIneqEqn} gives that $\hbar\le e$, so it is sufficient to show $\hbar\left(\Upsilon_k\left(\frac{\pi}{k+1}\right)\right)\ge\frac{\pi}{k+1}$ and $e\left(\Upsilon_k\left(\frac{\pi}{k+1}\right)\right)\le\frac{\pi}{k+1}$. The lower bound on $\hbar$ follows from the fact that $\omega(H_2(\mathbb{C}^3,\Upsilon_k))=\mathbb{Z}\left[\frac{\pi}{k+1}\right]$ in the monotone case, as computed in Proposition \ref{BrendelRelativeHomologyProp}. The upper bound on $e$ was proved in Proposition \ref{dispEnergyUpsilonProp}. This completes the proof.
\end{proof}

\subsection{Holomorphic disks in \texorpdfstring{$\mathbb{C}^3$}{ℂ³}}\label{holoDiskUpsilonSect}

\subsubsection{\texorpdfstring{$\hbar$}{ℏ} in the non-monotone case}\label{nonMonUpsilonHbarSubsect}

In this subsection we find the classes in $H_2(\mathbb{C}^3,\Upsilon_k)$ which are realized by $J_0$-holomorphic disks, without assuming monotonicity. The method we apply, of finding such classes by applying positivity of intersections, originates in the analysis of holomorphic disks on the Clifford torus in \cite[Theorem 9.1]{cho2003holomorphicdiscspinstructures}, which was adapted to the case of the Chekanov torus in \cite{chekanov2010notes} and \cite[Section 5]{auroux2007mirrorsymmetrytdualitycomplement} and was since used also in \cite{Vianna_2014,Oakley_2016} among others. During our calculation, we consider intersection of such holomorphic disks with a certain complex variety (denoted $\Sigma_F\subset\mathbb{C}^3$ below), whose construction is inspired by the construction of the Chekanov torus in \cite{knot-survey,auroux2007mirrorsymmetrytdualitycomplement}.

Define $F\colon\mathbb{C}^*\times\mathbb{C}^2\to\mathbb{C}$ by $(z_1,z_2,z_3)\mapsto z_1^{-k}z_2z_3$. This map is clearly invariant under the $T^2$-action \eqref{brendelT2actionEqn}, as is the map $q$ \eqref{eqn:qProjectionFromLevelSets}.

\begin{prop}\label{possibleClsUpsilonProp}
    The only classes in $H_2(\mathbb{C}^3,\Upsilon_k)$ which may contain $J_0$-holomorphic disks of Maslov index $2$ are $\alpha$ and $\beta_1-k\alpha+n\beta_2$ for $0\le n\le k$.
\end{prop}
\begin{proof}
For a proper holomorphic complex hypersurface $\Sigma$ in $\mathbb{C}^3$ disjoint from $\Upsilon_k$, and an oriented surface $Y\subset\mathbb{C}^3$ with boundary on $\Upsilon_k$, the intersection index $\Sigma\cdot Y$ is well defined. The three holomorphic complex hypersurfaces in the top row of the table below are indeed disjoint from $\Upsilon_k$. The table presents the intersection index of these hypersurfaces with the cycles $\alpha$, $\beta_1$, $\beta_2$; the rightmost column of the table gives the values of the Maslov class for these disks.
\[
\begin{tabular}{|c||c|c|c||c|} 
 \hline
 class & $z_1$-$z_3$ plane & $\Sigma_{12}$ & $\Sigma_F$ & $\mu$ \\
 \hline
 $\alpha$ & 0 & 0 & 1 & 2 \\
 $\beta_1$ & 0 & $k$ & $k$ & $2k+2$ \\
 $\beta_2$ & 1 & $-1$ & 0 & 0 \\
 \hline
\end{tabular}
\]
Here, $\Sigma_{12}$ is the $z_1$-$z_2$ plane, perturbed near the origin to become transverse to $\beta_1$ and remain disjoint from $\alpha$.
Additionally, \[\Sigma_F\coloneqq\scomp{(z_1,z_2,z_3)\in\mathbb{C}^3}{s z_1^k=z_2z_3}=F^{-1}(s)\cup\scomp{(0,z,0),(0,0,z)}{z\in\mathbb{C}}\] is a proper complex subvariety of $\mathbb{C}^3$, where \begin{equation*}
    s\coloneqq F(q^{-1}(\epsilon))\in\mathbb{C},
\end{equation*}
($F$ is constant on $q$'s fibers which are the $T^2$-orbits) and $\epsilon\in\interior{D}$, which is not real proportional to $w_0$ (this choice ensures that $\Sigma_F$ is transverse to $\beta_2$). Note that $\Sigma_F$ is disjoint from $\Upsilon_k$. Indeed, by the explicit expression \eqref{UpsilonExprEqn}, no point in $\Upsilon_k$ has a zero entry, so $\Sigma_F\cap\Upsilon_k=F^{-1}(s)\cap\Upsilon_k$. Now, \[F^{-1}(s)\cap\Upsilon_k=q^{-1}(\epsilon)\cap q^{-1}(\Gamma)=q^{-1}(\{\epsilon\}\cap\Gamma)=\emptyset.\]

The intersection index $\Sigma_F\cdot\beta_1=k$ is calculated as follows. We have $\Sigma_F=\{G=0\}$, where $G(z_1,z_2,z_3)=sz_1^k-z_2z_3$. Thus, the intersection index is given by the vanishing order at the origin of \[G\circ\beta_1(z)=s\left(\sqrt{1-k\lvert w_0\rvert^2}z\right)^k-\lvert w_0\rvert w_0z^k=\left(s\left(\sqrt{1-k\lvert w_0\rvert^2}\right)^k-\lvert w_0\rvert w_0\right)z^k,\] which is $k$ because $s=\frac{\lvert\epsilon\rvert\epsilon}{\left(\sqrt{1-k\lvert\epsilon\rvert^2}\right)^k}$ and $\epsilon\in\interior{D}$, while $w_0\in\partial D=\Gamma$. The other entries in the table are straightforward to check.

A class of Maslov index $2$ is of the form \[\alpha+m(\beta_1-(k+1)\alpha)+n\beta_2\] for $m,\,n\in\mathbb{Z}$. Assuming that this class contains a holomorphic curve, we have constraints on $m$ and $n$ coming from positivity of intersections. Considering the intersection number with the planes, we must have $0\le n\le mk$; and considering the intersection number with $\Sigma_F$, we must have $m\le1$. It follows that the only possibilities are $m=n=0$ and $m=1,\,0\le n\le k$.
\end{proof}

\begin{corollary}\label{UpsilonHlowerbdCor}
    $\hbar(\Upsilon_k)=\pi-ka$.
\end{corollary}
\begin{proof}
A class of Maslov index $2\ell$ is of the form \[\ell\alpha+m(\beta_1-(k+1)\alpha)+n\beta_2\] for $\ell,m,\,n\in\mathbb{Z}$. Assuming that this class contains a holomorphic disk, we have constraints on $\ell,\,m$ and $n$ coming from positivity of intersections. Considering the intersection number with the planes, we must have $0\le n\le mk$; and considering the intersection number with $\Sigma_F$, we must have $m\le\ell$.

This class has area $A\coloneqq\pi m+(\ell-mk-m)a$. In  case $\ell=m>0$, we have $A=m(\pi-ka)\ge\pi-ka$. Otherwise, $\ell>m\ge0$ and we have $A\ge a+m(\pi-ka)\ge a\ge\pi-ka$, since $\frac{\pi}{k+1}\le a$. So in either case, $A\ge\pi-ka$. Thus $\pi-ka\le\invH(\Upsilon_k,J_0)\le\invH(\Upsilon_k)$.

Finally, we apply Proposition \ref{hbarLSCprop} to bound $\hbar(\Upsilon_k)$ from above. The Lagrangians converging to our Lagrangian knot $\Upsilon_k$, come from its \emph{versal deformation} as computed in \cite[Proposition 4.4]{brendel2023local}.

Let $0<t\le1$. The $T^2$-quotient of the level set $\nu_k^{-1}(\pi,t)\subset\mathbb{C}^3$ yields a reduced space $\Sigma_t$ with an orbifold point and no puncture, as opposed to the reduced space in the construction of $\Upsilon_k$. Denote by $L_t\subset\nu_k^{-1}(\pi,t)$ the lift of a curve in $\Sigma_t$ bounding area $a$. By the above-mentioned versal deformation computation, $L_t\xrightarrow[t\to0]{C^\infty}\Upsilon_k$. Furthermore, $L_t$ can be Hamiltonianly isotoped to the product torus $T(\pi-ka,a+t,a)$. Thus, $\lim_{t\to0}\hbar(L_t)=\pi-ka$, and $\hbar(\Upsilon_k)\le\pi-ka$ by Proposition \ref{hbarLSCprop}.
\end{proof}

\begin{remark*}
    The above proof gives evidence that $\Upsilon_k(a)$, including the non-monotone case, is a \emph{positive Lagrangian submanifold} in the sense of \cite[Section 3.2]{Cho_2008} (except that Cho considers Lagrangians in a \emph{compact} symplectic manifold). This means that for some regular compatible almost complex structure $J$, all non-constant pseudo-holomorphic disks have positive Maslov index and that the evaluation map \eqref{evMapEqn} is a submersion for classes of pseudo-holomorphic disks with Maslov index $2$.
\end{remark*}

The following two subsections contain consequences of this subsection.

\subsubsection{Classification of Brendel tori}

In answer to \cite[Remark 1.1]{brendel2023local}, we have the following corollary. 
\begin{corollary}[Classification up to symplectomorphism of Brendel tori]\label{upsilonClassificCor}
    No two $\Upsilon_k(a)$ with distinct parameters $k,a$ are symplectomorphic.
\end{corollary}

\begin{proof}
    By \cite[Theorem A]{brendel2023local}, $\Upsilon_k(a)\ncong\Upsilon_{k'}(a')$ for $k\neq k'$ and any $a,a'$. Fix $k\ge2$ and let $\frac{\pi}{k+1}\le a,a'<\frac{\pi}{k}$. We show that
    \begin{equation}\label{EquivUpsilonsEqn}
    \Upsilon_k(a)\cong\Upsilon_k(a')    
    \end{equation}
    implies $a=a'$. Since $\invH$ is a symplectic invariant, \eqref{EquivUpsilonsEqn} implies \[\pi-ka=\invH(\Upsilon_k(a))=\invH\left(\Upsilon_k\left(a'\right)\right)=\pi-ka'\] and we're done.
\end{proof}

\begin{remark}
Lagrangians $L,L'\subset M$ are called \emph{area equivalent} if they are Lagrangian isotopic and the isomorphism $\Phi\colon H_2(M,L)\to H_2(M,L')$ induced by the isotopy, preserves the area class, i.e., $\Phi^*\omega_L=\omega_{L'}$.

The classification result above does not require the calculation of $\hbar(\Upsilon_k(a))$ because it can be proved using the simpler fact that $\Upsilon_k(a)$ is area equivalent to $T(\pi-ka,a,a)\subset\mathbb{C}^3$, and that the Chekanov lattices of area equivalent Lagrangian tori in $\mathbb{C}^n$ are equal.

In fact, by calculating $\hbar(\Upsilon_k)$ we have shown the following stronger classification result. Brendel's original construction is slightly more general than we explained above. Instead of performing symplectic reduction on the level set $\nu_k^{-1}(\pi,0)$, Brendel deals with a general level set $\nu_k^{-1}(a_1,0)$ for arbitrary $a_1>0$. Then, the reduced space has area $\frac{a_1}{k}$ and the rest of the construction goes through similarly where for $a=(a_1,a_2)$ and $a_2\in[\frac{a_1}{k+1},\frac{a_1}{k})$, $\Upsilon_k(a)$ is obtained by lifting a curve enclosing area $a_2$ from the reduced space. In this generality, $\hbar(\Upsilon_k(a))=a_1-ka_2$. A consequence of that, which does not follow from the Chekanov lattice, is that no two $\Upsilon_k(a)$ with distinct parameters are symplectomorphic, even when using the more general notation where $a=(a_1,a_2)$.
\end{remark}

\subsubsection{Correspondence with Auroux's tori}\label{AurouxSection}

Let us define the notion of \emph{pseudo-holomorphic disk count}, see \cite[Section 2.2]{Chekanov98} and \cite[Section 3]{auroux2007mirrorsymmetrytdualitycomplement}. Given a \emph{monotone} spin
Lagrangian $L\subset M$, a regular
almost complex structure $J$ and a relative class $c\in H_2(M,L)$ of minimal positive Maslov index $2$, consider the moduli space \[\mathcal{M}_1(L,c,J)\coloneqq\scomp{u\colon(B^2,\partial B)\to(M,L)}{\text{$J$-holomorphic, }[u]=c}/\operatorname{Aut}(B^2(\pi),1).\]
It is guaranteed to be compact by $L$'s monotonicity and the fact that $c$ has minimal Maslov index, which exclude bubbling or multiple
covering phenomena. Further, by the Riemann-Roch theorem, the moduli space is either empty or a smooth oriented manifold of dimension \[\dim{L}+\mu(c)-\dim\operatorname{Aut}(B^2(\pi),1)=\dim{L}.\] The pseudo-holomorphic disk count $n_c(L)\in\mathbb{Z}$ is then the degree of the evaluation map
\begin{equation}\label{evMapEqn}
\begin{aligned}
    \mathcal{M}_1(L,c,J)&\to L \\ [u]&\mapsto u(1).
\end{aligned}
\end{equation}

\begin{fact*}
For $\phi\in\operatorname{Symp}(M)$, we have $n_c(L)=n_{\phi_*c}(\phi(L))$. Also, $n_c(L)$ is independent of the choice of regular $J$.
\end{fact*}

\begin{theorem}[\cite{Auroux_monTorR6}]
    For any $k\ge0$, there exists a monotone Lagrangian torus $T_\textnormal{Au}(k)\subset\mathbb{R}^6$ having exactly $k+2$ Maslov $2$ classes in $H_2(\mathbb{R}^6,T_\textnormal{Au}(k))$ for which the pseudo-holomorphic disk count is nonzero. The sum of their counts is $2^k+1$. We call these Lagrangians \emph{Auroux's tori}.
\end{theorem}

Another corollary of the analysis of classes of pseudo-holomorphic disks with boundary on $\Upsilon_k$ is some evidence of a correspondence between Brendel's and Auroux's tori. We have shown in Proposition \ref{possibleClsUpsilonProp} that there can be at most $k+2$ Maslov $2$ classes in $H_2(\mathbb{C}^3,\Upsilon_k)$ that contain a holomorphic disk. The following proposition motivates our expectation that, in the monotone case, the sum of the disk counts of these classes is equal to $2^k+1$, cf. \cite[Proposition 5.7]{auroux2007mirrorsymmetrytdualitycomplement} for the case of monotone $T_\textnormal{PrCh}^2\subset\C P^2$. In turn, this suggests that a monotone $\Upsilon_k$ could be Hamiltonian isotopic to an appropriate scaling of $T_\textnormal{Au}(k)$.

\begin{prop}\label{upsilonHoloDisksIntersectionsProp}
    If the relative class $\beta_1-k\alpha+n_0\beta_2$ contains a $J_0$-holomorphic disk for some $0\le n_0\le k$, then $\beta_1-k\alpha+n\beta_2$ contains at least $\binom{k}{n}$ $J_0$-holomorphic disks passing through a fixed $x\in\Upsilon_k$ for all $0\le n\le k$.
\end{prop}
\begin{proof}
Fix $x\in\Upsilon_k$ and set $B\coloneqq B^2(\pi)$. Assume $u\colon(B,\partial B)\to(\mathbb{C}^3,\Upsilon_k)$ is $J_0$-holomorphic with $u(1)=x$, such that $[u]=\beta_1-k\alpha+n_0\beta_2$. Then each $u_j\colon B\to\mathbb{C}$ is holomorphic, $1\le j\le3$. The intersection $[u]\cdot\{z_2=0\}=n_0$ implies that $u_2$ has exactly $n_0$ zeroes \[\{w_j\}_{j=1}^{n_0}\subset\interior{B}\] in the interior of $B$. For fixed $w\in\interior{B}$, the holomorphic function $f_w(z)\coloneqq\frac{1-\conj{w}}{1-w}\cdot\frac{z-w}{1-\conj{w}z}$ on $B$ preserves $\partial B$, has $f_w(1)=1$, and has only the zero $w$, which is simple.

Then the holomorphic disk $d\colon B\to\C^3$ given by \[d(z)=\left(u_1(z),\frac{u_2(z)}{\prod_{j=1}^{n_0}f_{w_j}(z)},u_3(z)\prod_{j=1}^{n_0}f_{w_j}(z)\right),\] has boundary on $\Upsilon_k$, is in the class $\beta_1-k\alpha$ and takes the value $x\in\Upsilon_k$ for $z=1$.

Similarly, denote by $\{v_j\}_{j=1}^{k-n_0}\subset\interior{B}$ the zeroes of $u_3$. Denote $[p]\coloneqq\{1,2,\dotsc,p\}$, so that $[0]\coloneqq\emptyset$. Let $0\le n\le k$. By Vandermonde's convolution, there are \[\sum_{p=0}^n\binom{n_0}{p}\binom{k-n_0}{n-p}=\binom{k}{n}\] choices of $\{a_j\}_{j\in[p]}\subset\{w_j\}_{j=1}^{n_0}$ and $\{b_j\}_{j\in[n-p]}\subset\{v_j\}_{j=1}^{k-n_0}$ with distinct $a_j$s and $b_j$s. Each such choice gives a holomorphic disk \[z\in B\mapsto\left(d_1(z),d_2(z)\cdot\prod_{j\in[p]}f_{a_j}(z)\cdot\prod_{j\in[n-p]}f_{b_j}(z),\frac{d_3(z)}{\prod_{j\in[p]}f_{a_j}(z)\cdot\prod_{j\in[n-p]}f_{b_j}(z)}\right),\] having boundary on $\Upsilon_k$, in the class $\beta_1-k\alpha+n\beta_2$ and taking the value $x\in\Upsilon_k$ for $z=1$.
\end{proof}

To see that the assumption of the above proposition holds, that is, that a holomorphic disk is realized in \emph{some} class $\beta_1-k\alpha+n_0\beta_2$ for some $0\le n_0\le k$, start by looking at the non-monotone case. Let $\frac{\pi}{k+1}<a<\frac{\pi}{k}$. Since \[\hbar(\Upsilon_k)=\pi-ka<a=\omega(\alpha),\] assuming that $J_0$ is regular, a pseudo-holomorphic disk must exist in one of the classes $\beta_1-k\alpha+n\beta_2$ (otherwise, by Proposition \ref{possibleClsUpsilonProp}, the pseudo-holomorphic disk with smallest possible area is in the class $\alpha$\footnote{Indeed, by Gromov compactness, it suffices to take the supremum \eqref{hbarSupDefn} over \emph{generic} almost complex structures, since those are dense in $\mathcal{J}_\omega$.} which has $\hbar(\Upsilon_k)<a=\omega(\alpha)$, in contradiction to $\alpha$'s minimality). But then Proposition \ref{upsilonHoloDisksIntersectionsProp} implies that each class $\beta_1-k\alpha+n\beta_2$ has at least $\binom{k}{n}$ holomorphic disks passing through a fixed $x\in\Upsilon_k$. Finally, it's reasonable that these disks persist when taking $a\searrow\frac{\pi}{k+1}$, i.e., passing to the monotone case where the pseudo-holomorphic disk count is well-defined.

This adds up to at least $2^k$ disks in the classes $\beta_1-k\alpha+n\beta_2$. Conjecturally, another single disk in the class $\alpha$ gives a total of $2^k+1$, divided into $k+2$ classes, which aligns perfectly with the total count for Auroux's tori.

\subsubsection{Open questions}

In summary, we have that $e(\Upsilon_k)=\hbar(\Upsilon_k)=\frac{\pi}{k+1}$ in the monotone case. In the non-monotone case, we have: \[\pi-ka=\hbar(\Upsilon_k)\le e(\Upsilon_k)\le a.\]

We conjecture that all disks specified by Proposition \ref{possibleClsUpsilonProp} have positive holomorphic disk count in the monotone case (for a suitable choice of spin structure on $\Upsilon_k$). If so, this gives a new proof of the exoticity of $\Upsilon_k$, which does not rely on the displacement energy germ as done in \cite{brendel2023local}, but rather on the pseudo-holomorphic disk count of a fixed class being a symplectic invariant.

Here are the open questions at hand.
\begin{question}[= Question \ref{dispEOfnonMonUpsQst}]
What is $e(\Upsilon_k)$ in the non-monotone case?    
\end{question}

\begin{question}
    Is the monotone $\Upsilon_k$ Hamiltonian isotopic to an appropriate scaling of the torus $T_\textnormal{Au}(k)$ constructed in \cite{Auroux_monTorR6}? Their disk count seems to agree, as discussed above.
\end{question}

For $k=-1,0,1$, the tori $\Upsilon_k$ can be defined similarly. These turn out to be Chekanov's \emph{special tori} in $\mathbb{C}^3$ \cite{chekanovTorus}. For example, $\Upsilon_1$ is the product torus. The leftmost column below gives the $\Upsilon_k$'s conjectural number of classes with nonzero pseudo-holomorphic disk count.
\[
\begin{tabular}{|c||c|c|c} 
 \hline
 $k$ & $-1$ & $0$ & $1$  \\
 \hline
 $\Upsilon_k$ & $T^3_\text{Ch}$ & $S^1\times T^2_\text{Ch}$ & $T^3$ \\ \hline
 $\#\scomp{c}{n_c(\Upsilon_k)\neq0,\,\mu(c)=2}$ & $1$ & $2$ & $3$ \\
 \hline
\end{tabular}\cdots\begin{tabular}{c} 
 \hline
 $k\ge2$ \\
 \hline
 $\Upsilon_k$ \\ \hline
 $k+2$ \\
 \hline
\end{tabular}
\]

\begin{remark}[An alternative point of view on $\Upsilon_3$]
In the last paragraph of \cite[Section 4]{chekanov2010notes}, Chekanov and Schlenk suggest a construction of a monotone exotic Lagrangian torus in $\mathbb{C}^3(z_1,z_2,z_3)$. We explain here that this construction produces $\Upsilon_3$. The monotone Chekanov torus $T^3_\textnormal{Ch}\left(\frac{\pi}{4}\right)\subset\mathbb{C}^3$ fits inside $B^6(\pi)$\footnote{See the explanation above Equation \eqref{chekanovTorusProjSpaceChartEqn}.}. Thus, we may ``write'' $T^3_\textnormal{Ch}$ into $\mathbb{C}P^3$ by a Darboux chart, and ``read'' it back to $\mathbb{C}^3$ using another Darboux chart, to get the \emph{inverted torus} $T^3_\textnormal{inv}\subset\C^3$.
Explicitly, $T^3_\textnormal{inv}\coloneqq\varphi_{0,3}(T^3_\textnormal{Ch})$ where $\varphi_{0,3}$ is the transition map of two Darboux charts of $\C P^3$ with a line $\C P^1$ of area $\pi$, given by
\begin{align*}
    B^6(\pi)\cap\{z_3\neq0\}&\to B^6(\pi)\cap\{z_1\neq0\} \\ (z_1,z_2,z_3)&\mapsto\frac{\lvert z_3\rvert}{z_3}\left(\sqrt{1-\sum_{j=1}^3\lvert z_j\rvert^2},z_1,z_2\right).
\end{align*}
It is straightforward to see, using \eqref{highDimChekanovTorusExprEqn} and \eqref{UpsilonExprEqn}, that $T^3_\textnormal{inv}=\Upsilon_3\left(\frac{\pi}{4}\right)$.
\begin{figure}[H]
    \begin{minipage}{0.45\textwidth}
        \centering
    \begin{tikzpicture}[x=0.37cm,y=0.37cm,z=0.222cm,>=stealth]
\draw[->] (xyz cs:x=0) -- (xyz cs:x=9) node[above] {$x_1$};
\draw[->] (xyz cs:y=0) -- (xyz cs:y=9) node[right] {$x_3$};
\draw[->] (xyz cs:z=0) -- (xyz cs:x=-1.3,z=9) node[above] {$x_2$};

\draw[dashed] (xyz cs:x=0, y=0, z=0) -- (xyz cs:x=3.6, y=4, z=4);

\draw[red, very thick] (xyz cs:x=0.9, y=1, z=1) -- (xyz cs:x=2.25, y=2.5, z=2.5) node[below right] {\color{black}$T_\textnormal{Ch}^3\subset\mathbb{C}^3$};
\end{tikzpicture}
\end{minipage}
\hfill
\begin{minipage}{0.45\textwidth}
        \centering
     \begin{tikzpicture}[x=0.37cm,y=0.37cm,z=0.222cm,>=stealth]
\draw[->] (xyz cs:x=0) -- (xyz cs:x=9) node[above] {$x_1$};
\draw[->] (xyz cs:y=0) -- (xyz cs:y=9) node[right] {$x_3$};
\draw[->] (xyz cs:z=0) -- (xyz cs:x=-1.3,z=9) node[above] {$x_2$};

\draw[dashed] (xyz cs:x=9, y=0, z=0) -- (xyz cs:x=0, y=3, z=3);

\fill (xyz cs:x=0, y=3, z=3) circle (1pt);

\draw[red, very thick] (xyz cs:x=1.5, y=2.5, z=2.5) -- (xyz cs:x=6, y=1, z=1) node[above right] {\color{black}$\Upsilon_3\subset\mathbb{C}^3$};
\end{tikzpicture}
\end{minipage}
\caption{Chekanov and Schlenk's inversion trick.}
\end{figure}
\end{remark}


\appendix

\section{Displacement energy of a disk}

\begin{lemma}\label{lem:displaceDisk}
Let $A,\epsilon>0$. Then
\[e_{B^2(2A+\epsilon)}\left(B^2(A)\right)\le A.\]
\end{lemma}
\begin{proof}
The following well-known conjugation trick was pointed out to us by Octav Cornea. We include it for completeness.
Start by sending $B(A)$, by an area-preserving map $\phi\colon B(2A+\epsilon)\to B(2A+\epsilon)$, to a disk $B_+$ which is contained in the upper half of the ambient circle. This is a standard application of Moser's homotopy method. Denote the reflection of $B_+$ about the horizontal diameter of the ambient circle by $B_-$, so that $B_+\cap B_-=\emptyset$.

\begin{figure}[h]
\begin{minipage}{0.45\textwidth}
\centering
\begin{tikzpicture}[x=0.7cm,y=0.7cm]
\def\bigradius{1.5}
\def\outerradius{1.41421*\bigradius+0.2}

\draw[line width=0.5pt]
  (0:\bigradius) arc (0:360:\bigradius)
  -- cycle;

\draw[gray, line width=0.5pt]
  (0:\outerradius) arc (0:360:\outerradius)
  -- cycle;

\node at (0.8,-0.3) {\small$B(A)$};
\node at (-2.3,2.4) {\color{gray}$B(2A+\epsilon)$};
\end{tikzpicture}
    \end{minipage}
\hfill
\begin{minipage}{0.45\textwidth}
\centering
\begin{tikzpicture}[x=0.7cm,y=0.7cm]
    \def\bigradius{1.5}
    \def\outerradius{1.41421*\bigradius+0.2}

    \draw[gray, line width=0.5pt]
        (0:\outerradius) arc (0:360:\outerradius);

    \def\s{0.77377*0.95*0.83}  
    \def\fx{1.1}                
    \def\dy{0.42}               

    \draw[line width=0.5pt]
        (-2.6*\s*\fx,0+\dy)
        .. controls (-2.25*\s*\fx,2.1*\s+\dy) and (-1.05*\s*\fx,2.9*\s+\dy) ..
        (0,2.95*\s+\dy)
        .. controls (1.05*\s*\fx,2.9*\s+\dy) and (2.25*\s*\fx,2.1*\s+\dy) ..
        (2.6*\s*\fx,0+\dy)
        .. controls (2.4*\s*\fx,-0.8*\s+\dy) and (-2.4*\s*\fx,-0.8*\s+\dy) ..
        cycle;

    \draw[dashed, line width=0.5pt]
        (-2.6*\s*\fx,0-\dy)
        .. controls (-2.4*\s*\fx,0.8*\s-\dy) and (2.4*\s*\fx,0.8*\s-\dy) ..
        (2.6*\s*\fx,0-\dy)
        .. controls (2.25*\s*\fx,-2.1*\s-\dy) and (1.05*\s*\fx,-2.9*\s-\dy) ..
        (0,-2.95*\s-\dy)
        .. controls (-1.05*\s*\fx,-2.9*\s-\dy) and (-2.25*\s*\fx,-2.1*\s-\dy) ..
        cycle;

\node at (0.8,0.6) {\small$B_+$};
\node at (0.8,-0.6) {\small$B_-$};
\end{tikzpicture}
\end{minipage}
    \caption{The conjugation trick.}
\end{figure}

Let $\delta>0$. Take a Hamiltonian diffeomorphism $\psi$ on the ambient circle such that $\psi(B_+)=B_-$ and the Hofer norm satisfies $\lVert\psi\rVert\le A+\delta$. Such $\psi$ can be realized by a rotation or a vertical translation with magnitude scaled by a function of the horizontal direction, cut-off outside the trajectory of $B_+$.

We have $\psi(\phi(B(A)))\cap B_+=\psi(B_+)\cap B_+=B_-\cap B_+=\emptyset$. Thus, \[\phi^{-1}\circ\psi\circ\phi(B(A))\cap B(A)=\phi^{-1}\circ\psi\circ\phi(B(A))\cap\phi^{-1}(B_+)=\emptyset.\] That is, $\phi^{-1}\circ\psi\circ\phi$ displaces $B(A)$ from itself.

Finally, we have by the invariance of Hofer's norm to conjugation, \[\lVert\phi^{-1}\circ\psi\circ\phi\rVert=\lVert\psi\rVert\le A+\delta.\]
This is true for all $\delta>0$, so $e(B(A))\le A$.
\end{proof}

\section{The reduced space of Brendel tori}\label{UpsilonAppendix}

The standard moment map $\mu$ on $(\mathbb{C}^3,\omega_\text{std})$, generates the Hamiltonian $T^3$-action of coordinate-wise rotation. For an integer $k\ge2$, the moment map \[\nu_k=(\mu_1+k\mu_3,\mu_2-\mu_3)\] on the same space, generates a Hamiltonian action by a $2$-dimensional subtorus: 
\begin{equation}\label{brendelT2actionAppendixEqn}
    (\theta_1,\theta_2)\cdot(z_1,z_2,z_3)=\left(e^{2\pi i\theta_1}z_1,e^{2\pi i\theta_2}z_2,e^{2\pi i(k\theta_1-\theta_2)}z_3\right).
\end{equation}
The symplectic reduction by this $T^2$-action on the level sets of $\nu_k$ was described in \cite[Section 4.1]{brendel2023local}. We take a more explicit approach.

Remove all points with non-trivial stabilizer from the level set, \[Z_k\coloneqq\nu_k^{-1}(\pi,0)\setminus\mu^{-1}\left(\left\{(\pi,0,0),\left(0,\frac{\pi}{k},\frac{\pi}{k}\right)\right\}\right)\subset\mathbb{C}^3.\] Taking the $T^2$-quotient gives a projection $p_k\colon Z_k\to\Sigma_k$, and a symplectic reduced space $(\Sigma_k,\omega_k)$, whose symplectic structure is determined by \begin{equation}\label{sympRedEqn}
\begin{gathered}
    i_k^*\omega_\text{std}=p_k^*\omega_k, \\
    \begin{tikzcd}
    Z_k\ar[r,"i_k", hook] \ar[d,"p_k"] & \mathbb{C}^3 \\ \Sigma_k. &
\end{tikzcd}
\end{gathered}
\end{equation}

The smooth map given in polar coordinates on $\mathbb{C}^3$ \begin{align*}
    q\colon Z_k=\{r_1\neq0\neq r_2=r_3,1=r_1^2+kr_3^2\}&\to\mathbb{C}^* \\ (r_1 e^{2\pi i\theta_1},r_2 e^{2\pi i\theta_2},r_3 e^{2\pi i\theta_3})&\mapsto r_3 e^{2\pi i(\theta_2+\theta_3-k\theta_1)}
\end{align*} is clearly constant on the orbits of the $T^2$-action \eqref{brendelT2actionAppendixEqn}.

\begin{prop}
The map $q$ descends to a symplectomorphism \[\Tilde{q}\colon\Sigma_k\to B^*\left(\frac{\pi}{k}\right).\] That is, $q=\Tilde{q}\circ p_k|_{Z_k}$.
\end{prop}

The following commutative diagram contains our main maps:
\begin{center}
\begin{tikzcd}[row sep=large,column sep=large]
    Z_k\ar[r,"i_k", hook] \arrow[d,"p_k"'] \arrow[dr, "q"] & \mathbb{C}^3 \\
    \Sigma_k \arrow[r,"\sim"{sloped},"\tilde{q}"'] & B^*\left(\frac{\pi}{k}\right).
\end{tikzcd}
\end{center}

\begin{proof}
The map $\Tilde{q}$ is smooth since $q$ is smooth, see \cite[Proposition 5.20]{Lee03}. It is straightforward to verify that \[g(z)\coloneqq p_k\left(\sqrt{1-k\lvert z\rvert^2},\lvert z\rvert,z\right)\] is a smooth inverse of $\tilde{q}$. Note that $g$ is well defined and smooth for $z\in B^*\left(\frac{\pi}{k}\right)$.

We have proved that $\tilde{q}$ is a diffeomorphism. We now show that $\omega_k=\tilde{q}^*(\rho d\rho\wedge d\varphi)$ where $(\rho,\varphi)$ are polar coordinates on the punctured open ball.

We have $1=r_1^2+kr_3^2$ and $r_2=r_3$ on $Z_k$. Thus, $r_1dr_1+kr_3dr_3=0$ on $Z_k$. Now, using the coordinates $\rho e^{2\pi i\varphi}=\tilde{q}\left(r_1 e^{2\pi i\theta_1},r_2 e^{2\pi i\theta_2},r_3 e^{2\pi i\theta_3}\right)$,
\begin{align*}
i_k^*\omega_\text{std}&=i_k^*(r_1dr_1\wedge d\theta_1+r_2dr_2\wedge d\theta_2+r_3dr_3\wedge d\theta_3) \\
&=i_k^*(r_3dr_3\wedge(-kd\theta_1+d\theta_2+d\theta_3)) \\ &=p_k^*\tilde{q}^*(\rho d\rho\wedge d\varphi).
\end{align*}
This shows that the form $\tilde{q}^*(\rho d\rho\wedge d\varphi)$ on the reduced space $\Sigma_k$ satisfies \eqref{sympRedEqn}.
\end{proof}

For $a\in\big[\frac{\pi}{k+1},\frac{\pi}{k}\big)$, and a contractible loop $\Gamma\subset B^*\left(\frac{\pi}{k}\right)$ bounding area $a$, we have\[\Upsilon_k(a)\coloneqq q^{-1}(\Gamma)\subset\mathbb{C}^3.\]

\phantomsection
\addcontentsline{toc}{section}{References}

\printbibliography

\end{document}